\documentclass[reqno]{amsart}

\usepackage{amssymb,latexsym}
\usepackage{amsfonts}
\usepackage{dsfont}
\usepackage[pdftex]{hyperref}
\usepackage{color}
\usepackage{graphicx}
\usepackage{parskip}
\usepackage{mathrsfs}

\theoremstyle{plain}
\newtheorem{thm}{Theorem}[section]
\newtheorem{theorem}{Theorem}[section]
\newtheorem{lemma}{Lemma}[section]

\newtheorem{proposition}{Proposition}[section]
\newtheorem{defin}{Definition}[section]

\newtheorem{corollary}{Corollary}[section]
\numberwithin{equation}{section}

\newcommand{\p}{\partial}
\newcommand{\m}{\mathbb}

\newcommand{\g}{\gamma}

\newcommand{\D}{\Delta}

\newcommand{\rr}{\mathbb{R}}
\newcommand{\nn}{\mathbb{N}}

\newcommand{\ci}{\mathbb{T}}

\newcommand{\wh}{\widehat}
\newcommand{\ee}{\varepsilon}
\newcommand{\vp}{\varphi}
\newcommand{\s}{\sigma}
\newcommand{\ms}{\mathcal{S}}
\newcommand{\mf}{\mathscr{F}}

\begin{document}

\title[Cauchy Problem for the FW Equation]{Well-posedness and continuity properties of The Fornberg-Whitham Equation in Besov Spaces}
\author{John Holmes \& Ryan C. Thompson}
\date{May 31, 2016}
\address{Department of Mathematics\\Ohio State University\\Columbus, OH 43210}
\email{holmes.782@osu.edu}
\address{Department of Mathematics\\University of North Georgia\\ Dahlonega, GA 30597}
\email{ryan.thompson@ung.edu}

\begin{abstract}
In this paper, we prove well-posedness of the Fornberg-Whitham equation in Besov spaces $B_{2,r}^s$ in both the periodic and non-periodic cases.  This will imply the existence and uniqueness of solutions in the aforementioned spaces along with the continuity of the data-to-solution map provided that the initial data belongs to $B_{2,r}^s $.  We also establish sharpness of continuity on the data-to-solution map by showing that it is not uniformly continuous from any bounded subset of $B_{2,r}^s$ to $C([-T,T]; B^s_{2,r})$.  Furthermore, we prove a Cauchy-Kowalevski type theorem for this equation that establishes the existence and uniqueness of real analytic solutions and also provide blow-up criterion for solutions.
\end{abstract}

\maketitle

%
%
%
%
\section{Introduction}
We consider the Cauchy problem for the Fornberg-Whitham (FW) equation
\begin{equation}
\label{fw}
\begin{cases}
u_{xxt}-u_t+\frac92u_xu_{xx}+\frac 32uu_{xxx}-\frac32uu_x+u_x=0 \\ 
u(x,0)=u_0(x),
\end{cases}
\end{equation}
when $x \in \rr$ or $\ci$ and $ t \in \m{R}$.  This equation was first written down in 1967 by Whitham \cite{whitham} and again by Whitham and Fornberg \cite{fw} as a model for breaking waves.  We show that the above Cauchy problem for the FW equation is well-posed in both periodic and non-periodic  Besov spaces $B_{2,r}^s$ for particular parameter values of $r$ and $s$.  Furthermore, we demonstrate that the data-to-solution map is not uniformly continuous from any bounded subset of $B_{2,r}^s$ to $C([-T,T]; B^s_{2,r})$ for the aforementioned parameters, however it is H\"older continuous in the $B_{2,r}^\alpha$ topology for any $\alpha < s$.  We also establish the existence and uniqueness of real analytic solutions provided that the initial data is real analytic as well as formulating blow-up criterion of solutions. Our results hold in both periodic and non-periodic spaces, unless otherwise stated.

The FW equation can also be written in its non-local form
\begin{equation}
\label{fwnl}
u_t+\frac32uu_x=(1-\p_x^2)^{-1}\p_xu,
\end{equation}
where $\mathcal{F}[ {(1-\p_x^2)^{-1}f }]= \frac1{1+\xi^2} \wh f(\xi)$ for any test function $f$.  If we write the FW equation in this form, we see that it resembles a special case of a family of nonlinear wave equations
\[
u_t+auu_x=\mathscr{L}(u),
\]
where $a\in\m{R}$ and $\mathscr{L}$ is a linear operator with constant coefficients, which has been studied by multiple authors.  In both \cite{whitham} and \cite{fw}, the FW equation \eqref{fwnl} was compared with the Korteweg-de Vries (KdV) equation
\[
u_t+6uu_x=-\p_x^3u,
\]
which was first derived by Boussinesque \cite{b} in 1877 and then by Korteweg and de Vries \cite{kdv} in 1895.

Since 1834 with John Scott Russell's observation of solitary waves, we have known of their physical existence and it wasn't until the advent of the KdV equation via Boussinesque, Korteweg and de Vries that we were finally able to see their mathematical existence as well.  In fact, we have in the non-periodic case, that the solitary wave solutions of the KdV equation are of the form
\[
u(x,t)=\frac c2 \text{sech}^2\left(\frac{\sqrt c}{2} (x-ct)\right),
\]
where the constant $c$ is the wave speed.  Unfortunately, while we now had a model that generated Russell's solitary wave solutions, it was only appropriate for long wave behavior and not for short wave phenomena close to shore lines, which exhibit the property of wave breaking.  Furthermore, it did not produce solitary waves of greatest height with a sharp peaked crest which are now called peakons.  The goal now had become one of seeking out a model that generated such solutions that not only maintained dispersive qualities like KdV, but also exhibited the aforementioned property of wave breaking.

In 1967, while Whitham \cite{whitham} was exploring applications to water waves he wrote down the FW equation \eqref{fwnl} and noted that it produced the so-called peakon solutions with maximum amplitude of 8/9.  Then in 1978, Whitham and Fornberg \cite{fw} researched both numerical and theoretical results on the FW equation and wrote down the explicit peakon solutions
\[
u(x,t)=\frac89e^{-\frac12\left|x-\frac43t\right|},
\]
along with noting the wave breaking properties.  In fact, $8/9$ was found as a limiting behavior of the amplitude of the above exponential peakon.  From then on, many authors have researched and found other nonlinear wave equations that resemble that of FW.

In 1993, Camassa and Holm \cite{ch} produced the equation which now bears their names (Camassa-Holm equation (CH))
\begin{equation}
\label{ch}
(1-\p_x^2)u_t=uu_{xxx}+2u_xu_{xx}-3uu_x
\end{equation}
which appears in the context of hereditary symmetries studied by Fokas and Fuchssteiner \cite{ff}.  This equation was first written explicitly and derived from the Euler equations  by Camassa and Holm in \cite{ch}, where they also found its 
``peakon" traveling wave solutions.
These are solutions with a discontinuity in the first spacial derivative at its crest. The simplest one in the non-periodic case
is of the form $u_c(x, t)=ce^{-|x-t|}$, where $c$ is a positive constant.
The CH equation possesses many  other 
remarkable properties such as infinitely many conserved quantities, a bi-Hamiltonian structure and a Lax pair.
For more information about how CH arises in the context of hereditary symmetries we refer to \cite{ff}.  Concerning it's physical relevance, we refer the reader to the works by Johnson \cite{j1}, \cite{j2} and Constantine and Lannes \cite{cl}.

In 1998 the following equation appeared
\begin{equation}
\label{dp}
(1-\p_x^2)u_t=uu_{xxx}+3u_xu_{xx}-4uu_x.
\end{equation}
Degasperi and Procesi \cite{dp} discovered this equation as one of the three equations to satisfy asymptotic integrability conditions in the following family of equations
\begin{equation}
\label{fam}
u_t+c_0u_x+\g u_{xxx}-\alpha^2u_{txx}=(c_{1,n}u^2+c_{2,n}u_x^2+c_3uu_{xx})_x,
\end{equation}
where $\alpha, c_0, c_{1,n}, c_{2,n}, c_3 \in \m{R}$ are constants.  Other integrable members of the family \eqref{fam}
are the CH and the KdV equations.  Also, like the CH equation \eqref{ch}, the DP equation \eqref{dp} possesses peakon solutions of the form $u_c(x,t)=ce^{-|x-ct|}$.

Recently, Vladimir Novikov \cite{n}, in addition to 
several integrable equations with quadratic nonlinearities
like the CH and DP, generated 
about 10 integrable equations with cubic nonlinearities while investigating the integrability of Camassa-Holm type equations of the form 
\begin{equation}
\label{gen}
(1-\p_x^2)u_t=P(u,u_x,u_{xx},...),
\end{equation}
where $P$ is a polynomial of $u$ and its derivatives.  
One of these CH-type equations  with cubic nonlinearities, 
which happened to be a new integrable equation, is  the following
\begin{equation}
\label{ne}
(1-\p_x^2)u_t=u^2u_{xxx}+3uu_xu_{xx}-4u^2u_x.
\end{equation}
Equation \eqref{ne} is now called the Novikov equation (NE) and also possesses many properties exhibited by Camassa-Holm type equations.  One of which is the existence of peakon solutions of the form $u_c(x,t)=\sqrt{c}e^{-|x-ct|}$.  It is important to mention that the aforementioned equations are integrable; they possess infinitely many conserved quantities, an infinite hierarchy of quasi-local symmetries, a Lax pair and a bi-Hamiltonian structure.  In contrast, the FW equation is not integrable.

Our first result is concerned with the well-posedness for the FW equation \eqref{fw} in Besov spaces in both periodic and non-periodic cases.  We base this result primarily off of the methodology presented in Danchin's paper on the CH equation \cite{danchin}.  The result is stated as follows.

\begin{thm}
\label{wp}
If $s>\frac32$, $r \in (1,\infty)$ or $s \ge 3/2$, $r = 1$ and the initial data $u_0 \in B^s_{2,r}$, then there exists a lifetime $T>0$ and a unique solution $u \in C([-T,T]; B^s_{2,r})$ of the FW initial value problem, which depends continuously on the initial data. Additionally, we have the solution size and lifespan estimates
\begin{equation}
\label{sizeest}
\|u\|_{ B^s_{2,r}} \leq 2 \|u_0\|_{ B^s_{2,r}} , \quad \text{ and }\quad T\leq \frac{c}{ \|u_0\|_{ B^s_{2,r}}},
\end{equation}
where $c $ is a positive constant depending only upon $s$. Furthermore, when $r = \infty$, there exists a unique solution $u \in C([-T,T]; B^{s'}_{2,\infty})$, for any $s'<s$, which depends continuously upon the initial data and satisfies the same lifespan and energy estimates in 
$B^{s'}_{2,\infty}$.
\end{thm}

It is important to note that when $r=2$, this corresponds to the Sobolev space $H^s$ where well-posedness has been shown by Yin \cite{y} by applying Kato's semigroup approach.  Well-posedness in Sobolev spaces $H^s$ for $s>3/2$ was also shown by Holmes \cite{h2} where he employed a Galerkin type approximation argument as in Taylor \cite{t2}.  The same methodology was used by Himonas and Holliman for the CH equation \cite{hh3}, the DP equation \cite{hh1} and the Novikov equation \cite{hh2}.  

From our well-posedness result, we are also able to demonstrate that the dependence on the initial data is sharp, as summarized in the following theorem.

\begin{thm}
\label{nonunif}
If $s\geq \frac32$ and $r \in [1,\infty]$, then the data-to-solution map is not uniformly continuous from any bounded subset of $ B^s_{2,r} (\ci) $ to $C([-T,T]; B^s_{2,r}(\ci) )$. 
\end{thm}

To demonstrate this sharpness of continuity, sequences
of approximate solutions are constructed. Then the
actual solutions are found by solving the Cauchy problem with initial data given by the
approximate solutions at time $t = 0$.  The error produced by solving the Cauchy problem using approximate solutions is shown to be inconsequential.  This was motivated by the works of Himonas and Holliman \cite{hh1,hh2, hh3}, Himonas and Kenig \cite{hk}, Himonas, Kenig and Misio\l ek \cite{hkm}, and Himonas, Misio\l ek and Ponce \cite{hmp}, Thompson \cite{t} and Grayshan \cite{g} on CH-type equations.  Furthermore, Holmes investigated this result for the FW equation in Sobolev spaces in \cite{h2}.

Following our proof of Theorem \ref{wp}, we shall expand our range of continuity properties by considering weaker topologies on $B_{2,r}^s$.  Although the data-to-solution map is not uniformly continuous on $B_{2,r}^s$, we will show that this map is H\"{o}lder continuous if we choose a properly weakened topology.  This is summarized in the following result.

\begin{thm}
\label{holder}
If $s\geq \frac32 $, $r \in [1,\infty]$ and $u_0 \in B^s_{2,r}$, then for any $s-1\le \alpha<s$, the data-to-solution map is H\"older continuous from any bounded subset of $ B^{\alpha}_{2,r}$ to $C([-T,T]; B^\alpha_{2,r})$. 
\end{thm}

We also prove the uniqueness and existence of real analytic solutions.  Methodology employed here follows that of Himonas and Misio\l ek in \cite{hm4} where they show analyticity of solutions for the Camassa-Holm equation \eqref{ch}.  This is given by the following result.

\begin{thm}
\label{anal}
If the initial data $u_0$ is analytic on $\ci$ then there exists an $\ee >0$ and a unique solution $u(x,t)$ to the Cauchy problem for the FW equation \eqref{fwnl} that is analytic in both $x$ and $t$ on $\ci$ for all $t$ in $(-\ee,\ee)$.
\end{thm}

Finally, we prove the following blow-up criterion.  The proof of which will follow a similar argument to one presented in \cite{danchin}.

\begin{thm}
\label{blowup}
Let $u_0 \in B_{p,r}^s$ as in Theorem \ref{wp} and $u$ be the corresponding solution to the Cauchy problem \eqref{fw}.  Assume that $T^*$ is the maximal time of existence of the solution to the aforementioned Cauchy problem.  If
\[
T^*< \infty, \ \ \ \text{then} \ \ \int_0^{T^*}\|u_x\|_{L^\infty}=+\infty.
\]
\end{thm}

The paper is structured as follows.  In Section two, we provide some preliminary results and lemmas that are used throughout the rest of the document.  In sections three and four, we prove the well-posedness result as stated in Theorem \ref{wp} along with Theorem \ref{holder}.  Then in Section five, we prove nonuniform dependence upon the initial data as stated in Theorem \ref{nonunif}.  Finally, in Section 6 and 7 we prove analyticity of solutions and blow-up criterion as stated in Theorem \ref{anal} and Theorem \ref{blowup} respectively.

%
%
%
%
\section{Preliminaries}
In this section, we will recall some conclusions on the properties of Littlewood-Paley decomposition, the Besov spaces and the theory of the transport equation which can be seen in \cite{chem}, \cite{danchin, d2, d3}, \cite{vishik}.

\begin{lemma}
(Littlewood-Paley decomposition).  There exists a couple of smooth radial functions $(\chi,\vp)$ valued in $[0,1]$ such that $\chi$ is supported in the ball $B=\{\xi \in \m{R}^n,|\xi|\leq \frac43\}$ and $\vp$ is supported in the ring $C=\{\xi \in \m{R}^n,\frac34\leq|\xi|\leq \frac83\}$.  Moreover,
\[
\forall \xi \in \m{R}^n, \ \ \ \chi(\xi)+\sum_{q \in \m{N}}\vp(2^{-q}\xi)=1
\]
and
\[
supp \ \vp(2^{-q}\cdot) \ \cap \ supp \ \vp(2^{-q'}\cdot) = \emptyset, \ \ if \ \ |q-q'|\geq 2,
\]
\[
supp \ \chi(\cdot) \ \cap \ supp \ \vp(2^{-q}\cdot) = \emptyset, \ \ if \ \ |q|\geq 1.
\]
Then for $u \in \mathcal{S}'(\m{R})$ the nonhomogeneous dyadic blocks are defined as follows:
\begin{align*}
&\Delta_qu=0, \ \ if \ \ q \leq -2, \\
&\Delta_{-1}u=\chi(D)u = \mathscr{F}_x^{-1}\chi\mathscr{F}u, \\
&\Delta_qu=\vp(2^{-q}D)=\mathscr{F}_x^{-1}\vp(2^{-q}\xi)\mathscr{F}u, \ \ if \ \ q\geq 0.
\end{align*}
Thus $u=\sum_{q \in \m{Z}}\Delta_qu$ in $\mathcal{S}'(\m{R})$.
\end{lemma}
\textbf{Remark.}  The low frequency cut-off $S_q$ is defined by
\[
S_qu=\sum_{p=-1}^{q-1}\Delta u = \chi(2^{-q}D)u = \mf_x^{-1}\chi(2^{-q}\xi)\mf_xu, \ \ \forall q \in \nn.
\]
We can see that
\[
\D_p\D_qu\equiv 0, \ \ if \ \ |p-q|\geq 2,
\]
\[
\ \ \ \ \ \  \ \ \ \ \D_q(S_{p-1}u\D_pv)\equiv 0, \ \ if \ \ |p-q| \geq 5, \ \ \forall u,v \in \ms'(\rr)
\]
as well as
\[
\|\D_qu\|_{L^p} \leq \|u\|_{L^p}, \ \ \ \ \|S_qu\|_{L^p} \leq C\|u\|_{L^p}, \ \ \ \forall p \in [1,\infty]
\]
with the aid of Young's Inequality, where $C$ is a positive constant independent of $q$.

\begin{defin}
(Besov Spaces)
Let $s \in \rr$, $p,r \in [1 , \infty]$ and $u \in S'(\rr)$. 
Then we define the Besov space of functions as
$$
B^s_{p,r} = B^s_{p,r}(\rr) = \{ u \in S'(\rr) : \|u\|_{B^s_{p,r} }<\infty\} , 
$$
where
$$
\| u\|_{B^s_{p,r} }\dot = 
\begin{cases}
\left( \sum_{q \ge -1} (2^{sq} \| \Delta _q u \|_{L^p})^r \right)^{1/r} &\text{ if } 1 \le r< \infty
\\
\sup_{q\ge -1} 2^{sq}\|\Delta _q u \|_{L^p} &\text{ if } r= \infty.
\end{cases}
$$
In particular, $B_{p,r}^\infty=\bigcap_{s \in \rr}B_{p,r}^s$.
\end{defin}
\begin{lemma}
Let $s \in \rr$, $1\leq p,r,p_j,r_j \leq \infty$, $j=1,2$, then
\begin{enumerate}
\item  Topological properties:  $B_{p,r}^s$ is a Banach space which is continuously emedded in $\ms'(\rr)$.

\item  Density:  $C_c^\infty$ is dense in $B_{p,r}^s$ $\iff$ $p,r \in [1,\infty)$. 

\item  Embedding:  $B_{p_1,r_1}^s \hookrightarrow B_{p_2,r_2}^{s-n(\frac1p_1-\frac1p_2)}$, if $p_1 \leq p_2$ and $r_1 \leq r_2$.
\[
B_{p,r_2}^{s_2} \hookrightarrow B_{p,r_1}^{s_1} \ \ \ \text{locally compact if} \ \ s_1<s_2.
\]

\item  Algebraic properties:  $\forall s >0$, $B_{p,r}^s\cap L^\infty$ is a Banach algebra.  $B_{p,r}^s$ is a Banach algebra $\iff$ $B_{p,r}^s \hookrightarrow L^\infty \iff s>\frac1p$ or $(s \geq \frac1p \  \text{and} \ r=1)$.  In particular, $B_{2,1}^{1/2}$ is continuously embedded in $B_{2,\infty}^{1/2} \cap L^\infty$ and $B_{2,\infty}^{1/2} \cap L^\infty$ is a Banach algebra.

\item  1-D Moser-type estimates:  

(i)  For $s>0$,
\[
\|fg\|_{B_{p,r}^s} \leq C(\|f\|_{B_{p,r}^s}\|g\|_{L^\infty}+\|f\|_{L^\infty}\|g\|_{B_{p,r}^s}).
\]

(ii)  $\forall s_1 \leq\frac1p < s_2 (s_2 \geq \frac1p \ if \ r=1)$ and $s_1+s_2>0$, we have 
\[
\|fg\|_{B_{p,r}^{s_1}} \leq C\|f\|_{B_{p,r}^{s_1}}\|g\|_{B_{p,r}^{s_2}}.
\]

\item  Complex interpolation:
\[
\|f\|_{B_{p,r}^{\theta s_1+(1-\theta)s_2}} \leq \|f\|_{B_{p,r}^{s_1}}^\theta\|g\|_{B_{p,r}^{s_2}}^{1-\theta}, \ \ \ \forall f \in B_{p,r}^{s_1} \cap B_{p,r}^{s_2}, \ \ \forall \theta \in [0,1].
\]

\item  Real interpolation:  $\forall \theta \in (0,1), s_1>s_2, s=\theta s_1+(1-\theta)s_2$, there exists a constant $C$ such that
\[
\|u\|_{B_{p,1}^s} \leq \frac{C(\theta)}{s_1-s_2}\|u\|_{B_{p,\infty}^{s_1}}^\theta\|u\|_{B_{p,\infty}^{s_2}}^{1-\theta}, \ \ \forall u \in B_{p,\infty}^{s_1}.
\]
In particular, for any $\theta \in (0,1)$, we have that
\[
\|u\|_{B_{2,1}^{1/2}} \leq \|u\|_{B_{2,1}^{\frac32-\theta}} \leq C(\theta)\|u\|_{B_{2,\infty}^{1/2}}^\theta\|u\|_{B_{2,\infty}^{3/2}}^{1-\theta}.
\]

\item  Fatou lemma:  If $\{u_n\}_{n \in \nn}$ is bounded in $B_{p,r}^s$ and $u_n \to u$ in $\ms'(\rr)$, then $u \in B_{p,r}^s$ and 
\[
\|u\|_{B_{p,r}^s} \leq \liminf_{n \to \infty}\|u_n\|_{B_{p,r}^s}.
\]
\end{enumerate}
\end{lemma}

Let $f$ be a solution to the following linear transport equation
\begin{align}\label{l-t}
\begin{cases}
\p_t f + v \p_x f = F,
\\
f(x,0) = f_0(x).
\end{cases}
\end{align}
The following proposition can be found in \cite{danchin}, Proposition A.1.  
\begin{proposition}  \label{lt-est} 
Let  $s>\frac32$, $r \in (1,\infty)$ or $s \ge 3/2$, $r = 1$ and let $v$ be a function such that $\p_x v \in L^1 ( (0,T) ; B^{s-1} _{2,r} )$. Also, let $f_0 \in B^s_{2,r}$, $F \in L^1((0,T);B^s_{2,r})$ and suppose $f \in L^\infty((0,T); B^s_{2,r} )\cap C([0,T];S')$ is a solution to \eqref{l-t}.   
  Then for some $C = C(s)$ and  $ V(t) = \int_0^t \| \p_x v(\tau) \|_{B^{s-1}_{2,r}  } d\tau$ we have
\begin{align}
\| f(t) \|_{B^s_{2,r} } \le e^{CV(t)} \left( \| f_0\| _{B^s_{2,r} }  + C \int_0^t e^{-CV(\tau)}   \|F(\tau) \|_{B^s_{2,r} }  d\tau
\right).
\end{align}
Moreover, the above inequality holds when $r= \infty$ and for all $s'<s$.  
\end{proposition}

For the sake of brevity, from now on we will denote our Besov space cut-off functions $\vp \doteq \vp_q$.  All other cut-off functions $\vp$ are assumed to be of Schwartz class.

%
%
\section{Local Well-posedness on the Torus}
 
Well posedness is broken into four steps. First we show existence of a solution on lifepspan $T$ and energy estimate, then we show that the solution is unique, and finally we show that the solution depends continuously upon the initial  data. 

\subsection{Existence and lifespan}
\begin{theorem}\label{existence}
Suppose  $s>\frac32$, $r \in (1,\infty)$ or $s \ge 3/2$, $r = 1$   and $u_0 \in B^s _{2,r}$. Then there exists a time 
$$
T   \approx \frac{ 1 }{  \|u_0 \|_{B^s _{2,r} }   },
$$
 such that the Cauchy problem for the FW equation has a solution $u \in C([0, T]; B^s _{2,r} ) \cap C^1 ([0, T]; B^{s-1} _{2,r})$ and the solution has the bound 
$$
 \sup_{ t \in [0,T] }  \| u(t) \|_{B^s _{2,r} }   \le 2 \| u_0 \|_{B^s _{2,r} }  .
$$
\end{theorem}

To show existence of a solution, we will find a sequence of approximate solutions from which we may extract a sub-sequence which converges to a solution of FW.

\subsubsection{Proof of Theorem \ref{existence}}
 We have from the Besov space embedding theorem  that $B^s _{2,r} \hookrightarrow C^1 $ (in particular, in this case for all $\s \ge s-1$, $B^{\s} _{2,r}$ is an algebra). 
 We will consider the following sequence of smooth functions $\{ u_{n} (x,t)\} _{n\ge 0}$, (where the subscript denotes an indexation) with the initial function $u_0(x,t) = 0$ solving the Cauchy problem for the linear transport equation
\begin{align}\label{transport_eq}
\begin{cases}
\p_t u_{n+1}  +  \frac32 u_n  \p_x u_{n+1} = \Lambda^{-1} [\p_x u_n ]\\
u _{n+1} (x,0) =  J_{n+1} u_0(x),
\end{cases}
\end{align}
where $J_{n+1}$ is a  mollifier.

Our first task is to show that the solutions to the above linear transport equation are uniformly bounded on a common lifespan. 
Applying the first estimate from Proposition \ref{lt-est}, we have 
\begin{align}
\| u^{n+1 } (t) \|_{B^s _{2,r} } \le e^{CU_{n} (t)} \left(    \|u_0  \| _{B^s _{2,r} }  +C  \int_0^t e^{-CU_n(\tau)}   \|  \Lambda^{-1} [\p_x u_n ](\tau)\|_{B^s _{2,r} }  d\tau
\right),
\end{align} 
 where $U_n (t) = \frac32 \int_0^t \| \p_x u_n(\tau) \|_{B^{s-1} _{2,r}  } d\tau$. 
We estimate have 
\begin{align}
 \|  \Lambda^{-1} [\p_x u_n ]\|_{B^s _{2,r} }  
\le  \|    u _n   \|_{B^{s} _{2,r} }  ,
\end{align} 
from which we conclude
\begin{align}\label{un-energy}
\| u_{n+1 } (t) \|_{B^s _{2,r} } \le e^{CU_{n} (t)} \left(    \|u_0  \| _{B^s _{2,r} }  +C  \int_0^t e^{-CU_n(\tau)}  \|    u_n(\tau) \|_{B^{s} _{2,r}}  d\tau
\right),
\end{align}

\begin{lemma}[Lifespan and Energy]
There exists a minimum  lifespan 
$$
T = \frac{1}{ 4  C   \|u_0 \|_{B^s _{2,r} }   } ,
$$
 such that for all $n \in \mathbb{Z}^+$ and for all $t\in [0,T]$
$$
\| u_n (t) \|_{B^s _{2,r} }  \le \frac{    \| u_0 \|_{B^s _{2,r} }  }{  1 -  2C     \|u_0 \|_{B^s _{2,r} }    t   }  \le  2 \| u_0 \|_{B^s _{2,r} }    . 
$$
 \end{lemma}
\begin{proof}
We will prove this claim by induction. 
The base case is trivially true, therefore, plugging the inductive hypothesis into \eqref{un-energy} we obtain 
\begin{align}\label{un-energy-2}
\| u_{n+1 } (t) \|_{B^s _{2,r} } \le e^{CU_{n} (t)}   \|u_0 \| _{B^s _{2,r} }  +C  \int_0^t e^{CU_{n} (t)} e^{-CU_n(\tau)}     \frac{  \| u_0 \|_{B^s _{2,r} }  }{ 1 -  2C      \|u_0 \|_{B^s _{2,r} } ^k  \tau }   d\tau
 .
\end{align} 
By the definition of $U_n$, and the inductive hypothesis 
$$
U_n(t) \le   \int_0^t  \frac{  \| u_0 \|_{B^s _{2,r} }  }{ 1 -2 C      \|u_0 \|_{B^s _{2,r} }    \tau   } d\tau = \frac{-1}{ 2C} \ln(1   -2C     \|u_0 \|_{B^s _{2,r} }   t )
$$ 
Using the above estimate  we obtain the  estimate 
\begin{align} \label{3.6}
&e^{CU_{n} (t)}  \le \frac{1}{   { ( 1- 2C   \|u_0 \|_{B^s _{2,r} }  t )^{1/2} } }   .
\end{align}
Similarly, we estimate the difference
$$
U_n(t) -U_n(\tau)\le   \frac{-1}{ 2C} \ln(1   -2 C     \|u_0 \|_{B^s _{2,r} }   t ) + \frac{1}{2C} \ln(1   -2 C     \|u_0 \|_{B^s _{2,r} }   \tau ),
$$ 
which
yields the estimate 
\begin{align} \label{3.7}
e^{C U_n(t)} e^{ -CU_n(\tau) }  \le 
\left(  \frac{1-2 C  \|u_0 \|_{B^s _{2,r} }  \tau}{  { 1-2C  \|u_0 \|_{B^s _{2,r} }  t  } }  \right)^{1/2}  .
\end{align}

Using inequality \eqref{3.6} on the first term, and estimate \eqref{3.7} on the second term, we obtain \eqref{un-energy-2} is bounded by 
\begin{align}\label{un-energy-5} 
 \| u_{n+1 } (t) \|_{B^s _{2,r} } 
& \le 
\frac{ \|u_0 \|_{B^s _{2,r} }  }{  { 1- 2C     \|u_0 \|_{B^s _{2,r} } ^kt  } }   ,
\end{align} 
which completes the inductive step.
 \end{proof}
Next  we will extract a subsequence of $\{ u_n\}$ which converges to a solution of the FW equation. 

\noindent
{\bf Step 1.} 
Our first step will be to show that there exists a sub-sequence $\{u_{n_j}\}$ of $\{u_n\}$ which converges in $L^\infty( [0,T]; B^s _{2,r} )$. 
Since the family $\{ u_n\}$ is bounded (by the previous energy estimate) in $C([0,T]; B^s _{2,r} )$, we may apply the Fatou property in Besov spaces (See Danchin Theorem 2.72), and conclude that there is a sub-sequence, $\{ u_{n_j}\}$, which converges to a limit $u$, and 
$$
\| u\| _{B^s_{2,r} } \le C \liminf_{n_j\rightarrow \infty} \|u_{n_j}\|_{B^s_{2,r} } .
$$
 
\noindent
{\bf Step 2.} 
Next we will employ Ascoli's theorem, to conclude that our element $u$ is in $ C( [0,T]; B^{s-1} _{2,r})$. 
Since the uniform limit of continuous functions is continuous, if we show that the sequence $u_n$ is equitontinuous in $ C( [0,T]; B^{s-1} _{2,r})$, we show that there exists a limit point $ u\in C( [0,T]; B^{s-1} _{2,r})$.
Let $t_1 $ and $t_2$ be in $[0,T]$. By the mean value theorem, 
$$\|
u_n(t_1)  - u_n(t_2) \|_{B^{s-1} _{2,r} }\le  |t_1 - t_2 | \sup_{t \in [0, T]} \| \p_t u _n \|_{B^{s-1} _{2,r} } .
$$
We have 
\begin{align}\label{3.17}
 \| \p_t u _n \|_{B^{s-1} _{2,r} }  \le \frac32 \| u_{n-1}   \p_x u_n \|_{B^{s-1} _{2,r} }  + \| \Lambda^{-1} [\p_x u]  \|_{B^{s-1} _{2,r} } 
\end{align}
Using ${B^{s-1} _{2,r} } $ is an algebra, we find
$$
 \| \p_t u _n \|_{B^{s-1} _{2,r} }  \le \frac32 \| u_{n-1}  \|_{B^{s-1} _{2,r} }   \| \p_x u_n \|_{B^{s-1} _{2,r} }  +  \|  u_{n-1} \|_{B^{s-1} _{2,r} }
$$
Hence, 
$$\|
u_n(t_1)  - u_n(t_2) \|_{B^{s-1} _{2,r} }\le  |t_1 - t_2 | 2 \left(\| u_{0}  \|_{B^{s} _{2,r}}  +\| u_{0}  \|_{B^{s} _{2,r}}  ^2 \right).
$$
Thus, $u_n$ is equicontinuous and we conclude that there exists a limit point of $u_{n_j} $, $u$, such that $u \in C([0,T]; B^{s-1} _{2,r})$.

\noindent
{\bf Step 3.} We now show that the solution $ u \in C([0,T] ; B^{s}_{2,r})$. 
Since $u \in L^\infty([0,T] ; B^{s}_{2,r})$, we must show that that for every fixed $ t \in (0, T)$
$$
\lim_{\tau \rightarrow t} \| u(\tau) - u(t) \|_{B^s_{2,r} } = 0. 
$$
We will show this limit by applying the Dominated convergence theorem. 
By definition of the Besov space norm, and using the solution is mean-zero we have 
\begin{align}
\| u(\tau) - u(t) \|_{B^s_{2 ,r} } ^ r  
&= \sum_{q = 0}^\infty 2^{sqr} \left(  \sum_{n \neq 0} \left| \varphi_q (n) \left(  \wh u(n,\tau) -\wh u(n,t) \right)  \right|^2 \right) ^{r/2} 
\\
& = 
\sum_{q = 0}^\infty 2^{sqr} \left(  \sum_{n \neq 0}   \varphi_q^2 (n) \left(  \wh u^2(n,\tau) -  2 \wh u  (n,\tau)    \wh u  (n,t)   +\wh u^2 (n,t) \right)   \right) ^{r/2} 
\\
& = 
\sum_{q = 0}^\infty 2^{sqr} \left( 
\|  \varphi_q \wh u(\tau)\|_{\ell^2} ^2+ \|  \varphi_q \wh u(t)\|_{\ell^2}^2 -2 \langle\varphi_q \wh u(\tau), \varphi_q \wh u(t)\rangle_{\ell^2}
\right)^{r/2} 
.
\end{align}
First notice that the function being summed is dominated
$$
2^{sqr} \left( 
\|  \varphi_q \wh u(\tau)\|_{\ell^2} ^2+ \|  \varphi_q \wh u(t)\|_{\ell^2}^2 -2 \langle\varphi_q \wh u(\tau), \varphi_q \wh u(t)\rangle_{\ell^2}
\right)^{r/2}  
\le 
2^{sqr}  \left( 4 \sup_{t' \in [0, T]} \|  \varphi_q \wh u(t')\|_{\ell^2}^2\right)^{r/2}  ,
$$
which is $q$-summable since $u\in L^\infty([0,T]; B^s_{2,r}) $. 
Thus, if for every $q$, we show the two limits (where we are considering functions of $q$)
\begin{align}
 &\langle\varphi_q \wh u(\tau), \varphi_q \wh u(t)\rangle_{\ell^2}  \xrightarrow[]{\tau\rightarrow t} \| \varphi_q \wh u(t)\|_{\ell^2}^2
\label{limit_2}
\\
&\|  \varphi_q \wh u(\tau)\|_{\ell^2}  \xrightarrow[]{\tau\rightarrow t} \|  \varphi_q \wh u(t)\|_{\ell^2}  
\label{limit_1} ,
\end{align}
then we may conclude that for each $q$
$$
\lim_{\tau \rightarrow t} 2^{sqr} \left( 
\|  \varphi_q \wh u(\tau)\|_{\ell^2} ^2+ \|  \varphi_q \wh u(t)\|_{\ell^2}^2 -2 \langle\varphi_q \wh u(\tau), \varphi_q \wh u(t)\rangle_{\ell^2}
\right)^{r/2} 
= 0,
$$
and we may conclude that by the Dominated convergence theorem for series
$$
\lim_{\tau \rightarrow t} \| u(\tau) - u(t) \|_{B^s _{2,r} } = 0. 
$$
{\bf Proof of \eqref{limit_2}.}
We have 
\begin{align}
D_1\dot =&\langle\varphi_q \wh u(\tau), \varphi_q \wh u(t)\rangle_{\ell^2} - \| \varphi_q \wh u(t)\|_{\ell^2}^2
= 
\langle\varphi_q ( \wh u(\tau) - \wh u(t)) , \varphi_q \wh u(t)\rangle_{\ell^2}
,
\end{align}
and therefore, the Cauchy-Schwarz inequality yields 
\begin{align}
D_1 & 
\le
\| \varphi_q ( \wh u(\tau) - \wh u(t)) \|_{\ell^2}  \| \varphi_q \wh u(t)\|_{\ell^2}
\\
& = 
\| \Delta_q  (  u(\tau) -  u(t)) \|_{L^2}  \| \Delta_q  u(t)\|_{L^2}.
\end{align}
Summing over $q$, and using the definition of Besov spaces, the above is bounded by 
\begin{align}
D_1 & 
\le 
\sum_{q = 0}^\infty \| \Delta_q  (  u(\tau) -  u(t)) \|_{L^2} \sum_{q = 0}^\infty  \| \Delta_q  u(t)\|_{L^2}
\\
& = 
\| u(\tau) -  u(t) \|_{B^{0}_{2,1} }  \|  u(t)\|_{B^{0}_{2,1} }.
\end{align}
Since $s-1 >0$, we may apply the Besov space embedding theorem, to conclude 
\begin{align}
D_1 & 
\le
\| u(\tau) -  u(t) \|_{B^{s-1}_{2,r} }  \|  u(t)\|_{B^{s}_{2,r} } \le 2  \|  u_0\|_{B^{s}_{2,r} } \| u(\tau) -  u(t) \|_{B^{s-1}_{2,r} } .
\end{align}
From Step 2, we know that $\lim_{\tau \rightarrow t} \| u(\tau) -  u(t) \|_{B^{s-1}_{2,r} }   = 0$, and we complete the proof of \eqref{limit_2}. 

\noindent
{\bf Proof of \eqref{limit_1}.}
We see that 
\begin{align}
D_2 & \dot = 
\|  \varphi_q \wh u(\tau)\|_{\ell^2} ^2 - \|  \varphi_q \wh u(t)\|_{\ell^2}  ^2
\\
& = 
\sum_{n\neq 0}  \varphi_q^2  (n) \wh u^2(n,\tau) - \sum_{n\neq 0}  \varphi_q^2  (n) \wh u^2(n,t) 
\\
& = 
\sum_{n\neq 0}  \varphi_q  (n) \left( \wh u(n,\tau) - \wh u(n,t) \right)\varphi_q  (n) \left( \wh u(n,\tau) + \wh u(n,t)\right ).
\end{align}
Applying the Cauchy Schwarz inequality yields 
\begin{align}
D_2 & \le
 \| \varphi_q  \left( \wh u(\tau) - \wh u(t) \right) \|_{\ell^2} \| \varphi_q  \left( \wh u(\tau) + \wh u(t)\right )\|_{\ell^2} 
\\
& =
 \| \Delta_q  (  u(\tau) -  u(t)) \|_{L^2}  \| \Delta_q  (  u(\tau) +  u(t)) \|_{L^2} 
.
\end{align}
We sum over $q$ to obtain 
\begin{align}
D_2 & \le 
 \|    u(\tau) -  u(t)  \|_{B^{0}_{2,1} }   \|  u(\tau) +  u(t)  \|_{B^{0}_{2,1} } 
.
\end{align}
Using the embedding theorem for Besov spaces, as well as the solution size estimate for the second term yields
\begin{align}
D_2 & \le 
 4 \|u_0\|_{B^s _{2,r}}\|    u(\tau) -  u(t)  \|_{B^{s-1}_{2,r} }   
.
\end{align}
Taking the limit as $\tau \rightarrow t$, of  the above we conclude 
\begin{align}\lim_{\tau \rightarrow t} D_2  = 0 \Longrightarrow 
\|  \varphi_q \wh u(\tau)\|_{\ell^2}  \xrightarrow[]{\tau\rightarrow t} \|  \varphi_q \wh u(t)\|_{\ell^2}  ,
\end{align}
which completes the proof of continuity; i.e. $ u \in C([0,T] ; B^{s}_{2,r})$. 

%

\noindent
{\bf Step 4.} We now have enough information about $u$ to show that it is a  classical solution of the FW equation:
$$
\p_t u = - \frac32 u \p_x u + \Lambda^{-1} [ \p_x u ] .
$$ 

First, we see that $u_n \rightarrow u $ in $C([0,T]; C^1)$ since $B^{s}_{2,r}  \hookrightarrow C^1$, therefore the distributional derivatives, $\p_t u_n \rightarrow \p_t u$; i.e.
$$
\int_0^T u_n (t)\p_t \phi(t)dt  \rightarrow \int_0^T u(t ) \p_t \phi(t) dt   
$$ 
for $\phi \in\mathcal{S} ([0, T])$ and 
for $x$ almost everywhere. 
Now we will verify that the limit of $\p_t u_n$ is
 \begin{align}\label{3.37}
\p_t u_n \rightarrow  - \frac32 u \p_x u + \Lambda^{-1} [ \p_x u ] \in C([0, T]; C),
\end{align}
 which shows that $u$ is classically differentiable in $t$ and is a classical solution of FW. Indeed,  letting $\| u_n \| _{ C([0,T]; B^{s}_{2,r} )} <\rho$, we have 
 \begin{align}
\|   \Lambda^{-1} [ \p_x u_n ] -  \Lambda^{-1} [ \p_x u ] \|_{ C([0,T]; C)}  \le C_1(\rho) \| u_n - u\| _{ C([0,T]; B^{s}_{p,r} )}   \rightarrow 0
\\
\| u_n  \p_x u_{n+1} - u   \p_x u \|_{ C([0,T]; C)} \le  C_2(\rho) \| u_n - u\| _{ C([0,T]; B^{s}_{2,r} )}   \rightarrow 0,
\end{align}
  where we used the Besov space embedding theorem $(B^{s-1}_{p,r} \hookrightarrow C)$. Thus, taking the limit as $n \rightarrow \infty$ we obtain \eqref{3.37}. 

Since $u$ is differentiable in $t$, and satisfies the FW equation, we can measure in what space $\p_t u$ exists in.  We observe from inequality  \eqref{3.17}, by replacing $\p_t u_n$ with $\p_t u$, that $u \in C^1 ([0,T]; B^{s-1} _{2,r}$), thus we have shown existence of a solution to the Cauchy problem for the FW equation 
 $$
u \in C ([0,T]; B^{s} _{2,r}) \cap C^1 ([0,T]; B^{s-1} _{2,r}).
$$ 

\subsection{Uniqueness} The proof of uniqueness follows from the following energy estimate. 
Next proposition illustrates how a change in initial data reflects on the solution. Uniqueness and continuity of the solution map follow from this result. 
 \begin{proposition}\label{lip-dep}
 For $1< r <\infty$ and $s>3/2$, or $r = 1$ and $s\ge 3/2$; given two solutions $u,v \in  C([0, T]; B^s_{2,r})\cap C^1([0,T]; B^{s-1} _{2,r} )$ of FW with initial data $u_0, v_0 \in B^s_{2,r}$ we have
 \begin{equation}\label{lip-s-1}
\sup_{t\in [0, T]} \|u(t)-v(t) \|_{B^{s-1} _{2,r} } \le \|u_0-v_0\|_{B^{s-1} _{2,r} }
 \end{equation}
 \end{proposition}
\begin{proof}
Let $w=u-v$. Then
\begin{equation}
\partial_t w+ \frac32 v  \partial_x w = -\frac32  (u+v) w \partial_x u+ \Lambda^{-1} \left[ \p_x w\right].
\end{equation}
 This is a transport equation, and by Proposition \ref{lt-est}, followed by the algebra property for Besov spaces, we have
\begin{equation}
\|w\|_{B^{s-1}_{2,r}}\leq \|w_0\|_{B^{s-1}_{2,r}} e^{c\int_0^t I(\tau)d\tau}+C  \int_0^t e^{c\int_{\tau}^t I(\tau')d\tau'} \|w(\tau)\|_{B^{s-1}_{2,r}} d\tau
\end{equation}
where $I(t)=\| \partial_x v\|_{B^{s-1}_{2,r}}$, 
from which we obtain
\[ \frac{d}{dt}\left( \|w\|_{B^{s-1}_{2,r}}e^{-c\int_0^t I(\tau)d\tau} \right)\leq C   e^{-c\int_0^{t} I(\tau)d\tau} \|w(t)\|_{B^{s-1}_{2,r}} ,
\]
with $\|w\|_{B^{s-1}_{2,r}}|_{t=0}=\|w_0\|_{B^{s-1}_{2,r}}$. By Gronwall's inequality we obtain
\begin{align}  
\|w\|_{B^{s-1}_{2,r}} 
& \notag 
\leq 
\|w_0\|_{B^{s-1}_{2,r}} e^{ {R}t} . \qedhere
\end{align} 
\end{proof}

\subsection{Continuous dependence}
Let $u_0 \in  B^{s}_{2,r} $ and let $\{u_{0,m}\}_{m=0}^\infty $ be a sequence in $ B^{s}_{2,r} $ such that $u_{0,m} \rightarrow u_0$ in $ B^{s}_{2,r} $. 
If $u$ is the corresponding solution to the FW equation with initial data $u_0$, and $u_m$ are the solutions corresponding to the initial data $u_{0,m}$, then we will show 
$$
\lim_{m \rightarrow \infty} \|  u_m - u\|_{C([0,T];  B^{s}_{2,r} ) } = 0.
$$  
This is equivalent to showing that for any $\delta > 0$, there exists an integer $N>0$, such that if $m>N$
\begin{align} \label{delta1}
\|  u_m - u\|_{C([0,T];  B^{s}_{2,r} ) } < \delta.
\end{align}
To do this, we will add and subtract $u^\ee$ and $u^\ee_m$; solutions corresponding to mollified initial data $J_{1/\ee} u_0$ and $ J_{1/\ee} u_{0,m}$. Applying the triangle, 
we see that to prove inequality \eqref{delta1} it suffices to show 
\begin{align}
&\|  u_m -u_m^\ee   \|_{C([0,T];  B^{s}_{2,r} ) }< \delta/3 \\
&\|  u_m^\ee -u^\ee   \|_{C([0,T];  B^{s}_{2,r} ) } < \delta/3\\
  & \|  u^\ee -u   \|_{C([0,T];  B^{s}_{2,r} ) } < \delta/3.
\end{align}

\subsubsection{Estimating $u^\ee -u$ and $ u_m -u_m^\ee $.}
The estimates for the first and last terms are identical. Therefore, we will show the estimate 
$$
\|  u^\ee -u   \|_{C([0,T];  B^{s}_{2,r} ) } < \delta/3.
$$
We add and subtract $u_n$, the approximate solution  to the transport equation, equation \eqref{transport_eq}.
$$
\|  u^\ee -u   \|_{C([0,T];  B^{s}_{2,r} ) } <\|  u^\ee -u_n   \|_{C([0,T];  B^{s}_{2,r} ) } +\|  u_n -u \|_{C([0,T];  B^{s}_{2,r} ) } 
$$
We know that $u = \lim_{n\rightarrow\infty} u_n$ in $C([0,T];  B^{s}_{2,r} )$, thus by choosing $n$ sufficiently large, 
$$
\|  u_n -u \|_{C([0,T];  B^{s}_{2,r} ) }  < \delta/ 6.
$$

For the first term, define $u_n^\ee$ to be the solution to the transport equation \eqref{transport_eq} with initial data $J_n J_{1/\ee} u_0(x)$. Then we have 
$$
\|  u^\ee -u_n   \|_{C([0,T];  B^{s}_{2,r} ) } \le \|  u^\ee -u_n^\ee   \|_{C([0,T];  B^{s}_{2,r} ) }  + \|  u^\ee _n-u_n   \|_{C([0,T];  B^{s}_{2,r} ) } 
$$
From our proof of existence of solutions to the Cauchy problem for the FW equation, we know that  $u_n^\ee \rightarrow u^\ee$. Thus,  by choosing $n$ sufficiently large, 
$$
\|  u^\ee -u_n   \|_{C([0,T];  B^{s}_{2,r} ) } \le \delta/12+ \|  u^\ee _n-u_n   \|_{C([0,T];  B^{s}_{2,r} ) } 
$$
Let $v_n^\ee = u_n^\ee-u_n$. Then by linearity $v_n^\ee$ is a solution to the transport equation, with initial data $v_n^\ee(x,0) = J_n J_{1/\ee} u_0(x)-J_n   u_0(x)$, we may apply the estimate found in Proposition 1 to obtain
\begin{align}
\| v_n^\ee (t) \|_{B^s_{2,r} } \le    \| J_n J_{1/\ee} u_0(x)-J_n   u_0(x)\| _{B^s_{2,r} } <\delta/12
\end{align}
by choosing $1/\ee$ sufficiently large. 

\subsubsection{Estimating $u_m^\ee -u^\ee$.}We will now prove $\|  u_m^\ee -u^\ee   \|_{C([0,T];  B^{s}_{2,r} ) } < \delta/3$. For notational convenience, we will suppress the $\ee$ superscripts. However, we note that $u$ and $u_m$ are in $C([0,T];B^{s+1}_{2,r})$ since the mollified initial data is $B^{s+1}_{2,r}$. 
We let $v = u_m - u$, and $w = u_m + u$ to obtain 
\begin{align}
\p_t v  + w \p_x v 
& \notag= - v \p_x w + \Lambda^{-1} [ \p_x v] .
\end{align}
Using the energy estimate, \eqref{lip-s-1}, we obtain 
\begin{align}
\|  v  \|_{B^s_{2,r}  } \lesssim \|  v_0  \|_{B^s_{2,r}  }  = \| u_m^\ee(0) -u^\ee (0)  \|_{B^s_{2,r}  } \le  \| u_m(0) -u (0)  \|_{B^s_{2,r}  } 
\end{align}
Taking $m>N$ sufficiently large, we obtain 
\begin{align}\label{v_est}
\|u_m^\ee -u^\ee\|_{B^s_{2,r}  } 
  <\delta/3,
\end{align}
which completes the proof of continuous dependence for the data-to-solution map.

\subsection{H\"{o}lder dependence in a weaker topology}
The proof of uniquenss is found by showing the data-to-solution map is Lipshiz continuous from $B^{s-1}_{2,r}$ to $C([-T, T]; B^{s-1}_{2,r})$. 
By interpolation, we can obtain that the data-to-solution map is H\"older continuous if $s-1<\s<s$ with H\"older exponent $s-\s$.
\begin{corollary}
Let $r>1$ and $s>3/2$ or $r = 1$ and $s\ge 3/2$. If $\s \in \rr$ such that $s-1 <\s<s$, then the data-to-solution map is H\"older continuous from $B^\s_{2,r}$ to $C([0, T]; B^\s_{2,r})$.
\end{corollary}
\begin{proof}
The proof of this follows directly from the energy estimate in $B^{s-1}_{s.r}$ and interpolation.
\end{proof}

%
%
\section{Local Well-posedness on the Real Line}
Here we can duplicate the proofs for uniqueness and continuity of the flow map as well as the preliminary estimates and lifespan without any difficulty.  The trouble lies in the existence of a solution on the real line.  Furthermore, we must change our Friedrichs mollifiers to be non-periodic and thus we fix $j \in \mathcal{S}(\rr)$ such that $\hat{j}(0)=1$.  Then we may define $j_m(x)=mj(mx)$.  Therefore, we have our Friedrichs mollifier to be $J_m$ such that
\[
J_mf \doteq j_m \star f.
\]
Now we may proceed in the proof of existence.

\noindent
{\bf Step 1.} 
Our first step will be to show that there exists a sub-sequence $\{u_{n_j}\}$ of $\{u_n\}$ which converges in $L^\infty( [0,T]; B^s _{2,r} )$. 
Since the family $\{ u_n\}$ is bounded (by the previous energy estimate) in $C([0,T]; B^s _{2,r} )$, we may apply the Fatou property in Besov spaces (See Danchin Theorem 2.72), and conclude that there is a sub-sequence, $\{ u_{n_j}\}$, which converges to a limit $u$, and 
$$
\| u\| _{B^s_{2,r} } \le C \liminf_{n_j\rightarrow \infty} \|u_{n_j}\|_{B^s_{2,r} } .
$$
Taking the supremum over $t$ we obtain $u \in L^\infty([0,T];B_{2,r}^s)$.

Relabeling $\{u_{n_j}\}$ as $\{u_n\}$ we will refine this sequence to converge to $u \in C( [0,T]; B^{s-1} _{2,r})$.  Also, throughout the rest of the proof for existence, we shall apply a cut-off function $\vp \in \ms(\rr)$ to each of our terms to guarantee convergence in particular topologies along the real line.

\noindent
{\bf Step 2.} 
Next we will employ Ascoli's theorem, to conclude that our element $u$ is in $ C( [0,T]; B^{s-1} _{2,r})$.  First we take note of the following theorem that can be found in Danchin \cite{d3}.  
\begin{theorem}
Let $1 \leq p,r \leq \infty$, $s \in \rr$ and $\ee >0$.  For all $\vp \in C_c^\infty(\rr)$, the map $u \mapsto \vp u$ is compact from $B_{p,r}^{s+\ee}$ to $B_{p,r}^s$.
\end{theorem}
This tells us that the set $U(t)=\{\vp u_n(t)\}$ is precompact in $B^{s-1} _{2,r}$.  Since the uniform limit of continuous functions is continuous, if we show that the sequence $\vp u_n$ is equicontinuous in $ C( [0,T]; B^{s-1} _{2,r})$, we show that there exists a limit point $ \vp u\in C( [0,T]; B^{s-1} _{2,r})$.
Let $t_1 $ and $t_2$ be in $[0,T]$. By the mean value theorem, 
$$\|
\vp u_n(t_1)  - \vp u_n(t_2) \|_{B^{s-1} _{2,r} }\le  |t_1 - t_2 | \sup_{t \in [0, T]} \| \p_t u _n \|_{B^{s-1} _{2,r} } .
$$
We have 
\begin{align}\label{3.17}
 \| \p_t u _n \|_{B^{s-1} _{2,r} }  \le \frac32 \| u_{n-1}   \p_x u_n \|_{B^{s-1} _{2,r} }  + \| \Lambda^{-1} [\p_x u]  \|_{B^{s-1} _{2,r} } 
\end{align}
Using ${B^{s-1} _{2,r} } $ is an algebra, we find
$$
 \| \p_t u _n \|_{B^{s-1} _{2,r} }  \le \frac32 \| u_{n-1}  \|_{B^{s-1} _{2,r} }   \| \p_x u_n \|_{B^{s-1} _{2,r} }  +  \|  u_{n-1} \|_{B^{s-1} _{2,r} }
$$
Hence, 
$$\|
\vp u_n(t_1)  - \vp u_n(t_2) \|_{B^{s-1} _{2,r} }\le  |t_1 - t_2 | 2 \left(\| u_{0}  \|_{B^{s} _{2,r}}  +\| u_{0}  \|_{B^{s} _{2,r}}  ^2 \right).
$$
Thus, $\vp u_n$ is equicontinuous and we conclude that there exists a limit point of $\vp u_{n_j} $, $\vp u$, such that $\vp u \in C([0,T]; B^{s-1} _{2,r})$.

\noindent
{\bf Step 3.} We now show that the solution $ u \in C([0,T] ; B^{s}_{2,r})$. 
Since $u \in L^\infty([0,T] ; B^{s}_{2,r})$, we must show that that for every fixed $ t \in (0, T)$
$$
\lim_{\tau \rightarrow t} \| u(\tau) - u(t) \|_{B^s_{2,r} } = 0. 
$$
We will show this limit by applying the Dominated convergence theorem. 
By definition of the Besov space norm we have
\begin{align}
\| u(\tau) - u(t) \|_{B^s_{2 ,r} } ^ r  
&= \sum_{q \geq -1}2^{sqr}\|\D_q(u(\tau)-u(t))\|_{L^2}^r
\\
& = 
\sum_{q \geq -1} 2^{sqr}[\|\D_qu(\tau)\|_{L^2}^2+\|\D_qu(t)\|_{L^2}^2-2\langle\D_qu(\tau),\D_qu(t)\rangle_{L^2}]^r.
\end{align}
First notice that the function being summed is dominated
$$
2^{sqr}[\|\D_qu(\tau)\|_{L^2}^2+\|\D_qu(t)\|_{L^2}^2-2\langle\D_qu(\tau),\D_qu(t)\rangle_{L^2}]^r
\le 
2^{sqr}  \left( 4 \sup_{t' \in [0, T]} \|  \D_q u(t')\|_{L^2}^2\right)^{r}  ,
$$
which is $q$-summable since $u\in L^\infty([0,T]; B^s_{2,r}) $. 
Thus, if for every $q$, we show the two limits (where we are considering functions of $q$)
\begin{align}
 &\langle\D_qu(\tau),\D_qu(t)\rangle_{L^2}  \xrightarrow[]{\tau\rightarrow t} \|\D_qu(t)\|_{L^2}
\label{limit_2}
\\
&\|\D_qu(\tau)\|_{L^2}  \xrightarrow[]{\tau\rightarrow t} \|\D_qu(t)\|_{L^2}
\label{limit_1} ,
\end{align}
then we may conclude that for each $q$
$$
\lim_{\tau \rightarrow t} 2^{sqr}[\|\D_qu(\tau)\|_{L^2}^2+\|\D_qu(t)\|_{L^2}^2-2\langle\D_qu(\tau),\D_qu(t)\rangle_{L^2}]^r
= 0,
$$
and we may conclude that by the Dominated convergence theorem for series
$$
\lim_{\tau \rightarrow t} \| u(\tau) - u(t) \|_{B^s _{2,r} } = 0. 
$$
{\bf Proof of \eqref{limit_2}.}
We have 
\begin{align}
D_1\dot =&\langle\D_qu(\tau),\D_qu(t)\rangle_{L^2} - \| \D_q u(t)\|_{L^2}^2
= 
\langle\D_q(u(\tau)-u(t)),\D_q u(t)\rangle_{L^2}
,
\end{align}
and therefore, the Cauchy-Schwarz inequality yields 
\begin{align}
D_1 
& \le
\| \Delta_q  (  u(\tau) -  u(t)) \|_{L^2}  \| \Delta_q  u(t)\|_{L^2}.
\end{align}
Summing over $q$, and using the definition of Besov spaces, the above is bounded by 
\begin{align}
D_1 & 
\le 
\sum_{q = 0}^\infty \| \Delta_q  (  u(\tau) -  u(t)) \|_{L^2} \sum_{q = 0}^\infty  \| \Delta_q  u(t)\|_{L^2}
\\
& = 
\| u(\tau) -  u(t) \|_{B^{0}_{2,1} }  \|  u(t)\|_{B^{0}_{2,1} }.
\end{align}
Since $s-1 >0$, we may apply the Besov space embedding theorem, to conclude 
\begin{align}
D_1 & 
\le
\| u(\tau) -  u(t) \|_{B^{s-1}_{2,r} }  \|  u(t)\|_{B^{s}_{2,r} } \le 2  \|  u_0\|_{B^{s}_{2,r} } \| u(\tau) -  u(t) \|_{B^{s-1}_{2,r} } .
\end{align}
From Step 2, we know that $\lim_{\tau \rightarrow t} \| u(\tau) -  u(t) \|_{B^{s-1}_{2,r} }   = 0$, and we complete the proof of \eqref{limit_2}. 

\noindent
{\bf Proof of \eqref{limit_1}.}
We see that 
\begin{align}
D_2 & \dot = 
\|\D_qu(\tau)\|_{L^2}^2-\|\D_qu(t)\|_{L^2}^2
\\
& = 
\int_\rr(\D_qu(\tau))^2-(\D_qu(t))^2)dx
\\
& = 
\int_\rr\D_q(u(\tau)+u(t))\D_q(u(\tau)-u(t))dx.
\end{align}
Applying the Cauchy Schwarz inequality yields 
\begin{align}
D_2 & \le
 \| \Delta_q  (  u(\tau) -  u(t)) \|_{L^2}  \| \Delta_q  (  u(\tau) +  u(t)) \|_{L^2}.
\end{align}
We sum over $q$ to obtain 
\begin{align}
D_2 & \le 
 \|    u(\tau) -  u(t)  \|_{B^{0}_{2,1} }   \|  u(\tau) +  u(t)  \|_{B^{0}_{2,1} } 
.
\end{align}
Using the embedding theorem for Besov spaces, as well as the solution size estimate for the second term yields
\begin{align}
D_2 & \le 
 4 \|u_0\|_{B^s _{2,r}}\|    u(\tau) -  u(t)  \|_{B^{s-1}_{2,r} }   
.
\end{align}
Taking the limit as $\tau \rightarrow t$, of  the above we conclude 
\begin{align}\lim_{\tau \rightarrow t} D_2  = 0 \Longrightarrow 
\|\D_qu(\tau)\|_{L^2}  \xrightarrow[]{\tau\rightarrow t} \|\D_qu(t)\|_{L^2}  ,
\end{align}
which completes the proof of continuity; i.e. $ u \in C([0,T] ; B^{s}_{2,r})$.

\noindent
{\bf Step 4.} We now have enough information about $u$ to show that it is a  classical solution of the FW equation:
$$
\p_t u = - \frac32 u \p_x u + \Lambda^{-1} [ \p_x u ] .
$$ 

First, we see that $\vp u_n \rightarrow \vp u $ in $C([0,T]; C^1)$ since $B^{s}_{2,r}  \hookrightarrow C^1$, therefore the distributional derivatives, $\p_t \vp u_n \rightarrow \p_t \vp u$; i.e.
$$
\int_0^T \vp u_n (t)\p_t \phi(t)dt  \rightarrow \int_0^T \vp u(t ) \p_t \phi(t) dt   
$$ 
for $\phi \in\mathcal{S} ([0, T])$ and 
for $x$ almost everywhere. 
Now we will verify that the limit of $\p_t \vp u_n$ is
 \begin{align}\label{3.37}
\p_t \vp u_n \rightarrow  - \vp\frac32 u \p_x u + \vp\Lambda^{-1} [ \p_x u ] \in C([0, T]; C),
\end{align}
 which shows that $u$ is classically differentiable in $t$ and is a classical solution of FW. Indeed,  letting $\| \vp u_n \| _{ C([0,T]; B^{s}_{2,r} )} <\rho$, we have 
 \begin{align}
\|   \vp\Lambda^{-1} [ \p_x u_n ] -  \vp\Lambda^{-1} [ \p_x u ] \|_{ C([0,T]; C)}  \le C_1(\rho) \| \vp u_n - \vp u\| _{ C([0,T]; B^{s}_{p,r} )}   \rightarrow 0
\\
\| \vp u_n  \p_x u_{n+1} - \vp u   \p_x u \|_{ C([0,T]; C)} \le  C_2(\rho) \| \vp u_n - \vp u\| _{ C([0,T]; B^{s}_{2,r} )}   \rightarrow 0,
\end{align}
  where we used the Besov space embedding theorem $(B^{s-1}_{p,r} \hookrightarrow C)$. Thus, taking the limit as $n \rightarrow \infty$ we obtain \eqref{3.37}. 

Since $u$ is differentiable in $t$, and satisfies the FW equation, we can measure in what space $\p_t u$ exists in.  We observe from inequality  \eqref{3.17}, by replacing $\p_t \vp u_n$ with $\p_t \vp u$, that $u \in C^1 ([0,T]; B^{s-1} _{2,r}$), thus we have shown existence of a solution to the Cauchy problem for the FW equation 
 $$
u \in C ([0,T]; B^{s} _{2,r}) \cap C^1 ([0,T]; B^{s-1} _{2,r}).
$$

\section{Nonuniform dependence of the data-to-solution map on the torus}
In this section, we show that the data-to-solution map is not uniformly continuous. We will use the method of approximate solutions, and set 
$$
u^n (x,t) = \frac23  n^{-1} + n^{-s  } \cos(n x-t)\quad \text{ and } \quad v^n =- \frac23   n^{-1} + n^{-s  } \cos(n x+t).
$$
Now let $u_n$ and $v_n$ dentoe the actual solutions to the FW equation \eqref{fw} with $u^n$ and $v^n$ as initial data respectively.  We shall show the following properties concerning $u^n$, $v^n$ and the corresponding solutions $u_n$ and $v_n$. 
 
\begin{enumerate}
\item$
 \lim_{n\rightarrow \infty} \|u_{0,n}-v_{0,n}\|_{B^s_{2,r} } = 0
$ 
\item$  \|u_{n}(t)\|_{B^s_{2,r} }+ \|v_{n}(t)\|_{B^s_{2,r} } \le C\quad  \forall \quad 0\le t\le T
$ 
\item$  \|u_{n}(t)-u^n(t)\|_{B^s_{2,r} }+ \|v_{n}(t)-v^n(t)\|_{B^s_{2,r} } \le n^{-\alpha} ,\quad \alpha > 0, \quad  \forall \quad 0\le t\le T
$ 
\item$  \|u_{n}(t)-v_n(t)\|_{B^s_{2,r} } > \sin(t), \quad  \forall \quad 0< t\le T.$
\end{enumerate}
 
We will provide the details when $r= \infty$. In the case $r < \infty$, the proof is similar to the result contained in \cite{holmes-tiglay}.

\subsection{Convergence of initial data and common lifespan}
Taking the difference, we see that the first property is trivially satisfied. 
From our well--posedness result, the second property is satisfied if the initial data is bounded independent of $n$. We have
\begin{align}
\|u_{0,n}\|_{B^\gamma_{2,\infty} }   
&\le
 n^{-1}  +\| n^{-s } \cos(n x)\|_{B^\gamma_{2,\infty }} 
\\
&  
=  n^{-1} 
+ \sup_q  n^{-s  } 2^{\gamma q } \| \Delta_q  \cos(n x) \|_{L^2} ,
\end{align} 
Letting $\delta_{n}(m)$ be the indicator function, taking a value of $1 $ if $ m = n$ and $0$ otherwise, we have 
\begin{align}
 \| \Delta_q  \cos(nx) \|_{L^2} =  \frac\pi2 \left( \sum_{m=-\infty}^{\infty}    \varphi_q^2(m) \left| \delta_{n}(m) + \delta_{-n}(m)\right| ^2 \right)^{1/2} 
=  \frac\pi {2} \left(   2  \varphi_q^2(n)   \right)^{1/2} ,
\end{align}
where we used that $\varphi_q$ is symmetric. By the definition of $\varphi_q$, we have 
$$
 \varphi_q(n) = 
\begin{cases}
1 \quad \text{ if } \quad \frac{1}{\ln(2) } \ln(\frac38 n) \le q \le  \frac{1}{\ln(2) } \ln(\frac43 n) 
\\
0 \quad \text{ otherwise} .
\end{cases}  
$$
Hence, we   obtain 
\begin{align}\label{estimate_id} 
\|u_{0,n}\|_{B^\gamma_{2,\infty} }    
&  
\approx  
 n^{-1} 
+   n^{ \gamma-s  }  .
\end{align}

\subsection{Error bounds for $u_n - u^n$}  
We have that the difference between the exact and approximate solutions, $\mathcal{E} = u_n -  u^n $  satisfy
\begin{equation}\label{transport_eq_E} \begin{cases}
\partial_t \mathcal{E}  + \frac32 u^n  \partial_x \mathcal{E}  = -\frac32    \mathcal{E} \p_x u_n + \Lambda^{-1} \left[ \p_x \mathcal{E} \right]- R
\\
\mathcal{E}  (x,0) = 0,
\end{cases} 
\end{equation} 
where the residue $ R =\p_t u^n +    \frac32 u^n \p_x u^n-  \Lambda^{-1} \left[ \p_x  u^n  \right] $.
We shall show that for all $0<t<T$ 
$
\lim_{n\rightarrow \infty} \| \mathcal{E} (t)\| _{B^{s}_{2,r} } = 0,
$
which we shall accomplish   by showing that for  $\frac12 < \gamma < s-1$,
$
\lim_{n\rightarrow \infty} \| \mathcal{E} (t)\| _{B^{\gamma}_{2,r} } = 0.
$
Then we shall  interpolating between a particular  $s_1 < s$ and $s_2 >s$.

Since the error $\mathcal{E}$ satisfies the transport equation \eqref{transport_eq_E}, for $\gamma$ satisfying  $s-1> \gamma >1/2$ we may apply Proposition \ref{lt-est} to obtain 
 \begin{align} \label{5.16}
e^{-CV(t) }\| \mathcal{E} \| _{B^{\gamma}_{2,\infty} } \le   C \int_0^t e^{- CV(\tau) } \| F(\tau) \|_{B^{\gamma} _{2,\infty} } d\tau,
\end{align}
where 
$$
V(t) = \int_0^t \| \p_x v (\tau)\|_{B^{\gamma-1} _{2,\infty} } d\tau =   \int_0^t \| \p_x   u^n  \|_{B^{\gamma-1} _{2,\infty} } d\tau \lesssim n^{ \gamma - s } t
$$  
and 
$$
F(t) =  -\frac32    \mathcal{E} \p_x u_n + \Lambda^{-1} \left[ \p_x \mathcal{E} \right]- R.
$$
We differentiate  \eqref{5.16} with respect to $t$ and use 
\begin{align}\notag
\| F(\tau) \|_{B^{\gamma} _{2,\infty} }&  = \| -\frac32    \mathcal{E} \p_x u_n + \Lambda^{-1} \left[ \p_x \mathcal{E} \right]- R\| _{B^{\gamma} _{2,\infty} }   
\lesssim (1 + n^{ \gamma +1- s }  )
 \| \mathcal{E}  \|  _{B^{\gamma} _{2,\infty} }   + \| R \| _{B^{\gamma} _{2,\infty} } ,
\end{align} 
 to obtain the inequality 
 \begin{align}
\frac{d}{dt} \left[ e^{-CV(t) }  \| \mathcal{E}  \| _{B^{\gamma }_{2,\infty}} \right]\le   C  e^{- CV(t) }  (1 + n^{ \gamma +1- s }  ) \| \mathcal{E}  \|  _{B^{\gamma } _{2,\infty}} +  C  e^{- CV(t) }  \| R \|  _{B^{\gamma } _{2,\infty}} .
\end{align}
Using the bound on $V$ and setting $y = e^{-CV(t) }  \| \mathcal{E}  \| _{B^{\gamma }_{2,\infty}}$ , we obtain the ODE
 \begin{align}
\frac{d y}{dt} \le   C (1 + n^{ \gamma +1- s }  ) y +    C_{s,\gamma}   e^{C n^{ \gamma - s } t  }\| R \|  _{B^{\gamma } _{2,\infty}} .
\end{align}
Since $ \gamma + 1 < s$ we obtain 
 \begin{align} \label{E-est}
\| \mathcal{E}  \| _{B^{\gamma }_{2,\infty}}  \approx y \lesssim\| R \|  _{B^{\gamma } _{2,\infty}}  ,
\end{align}
i.e. the difference between actual and approximate solutions are bounded by the size of the residue, $R$, for the entire lifespan of the solutions. The following lemma yields the size estimate for the residue. 
\begin{lemma}
For all $t\in \rr$ and $0 < \gamma < s$, we have 
$\| R \|  _{B^{\gamma} _{2,\infty}}  <  n^{-\alpha} $ for some $\alpha = \alpha (s,\gamma) > 0$.
\end{lemma}
\begin{proof}
Recall, the residue is defined as 
$$
R =\p_t u^n +     \frac32  u^n  \p_x u^n -\Lambda ^{-1} \p_x u ^n .
$$
By the definition of the approximate solution $u^n(x,t)$, we have 
\begin{align*}
& \p_t  u^n (x,t) =n^{-s } \sin(n x-t), \quad \quad 
\p_x u^n (x,t) =   -n^{1 -s  } \sin(n x-t)
\end{align*}
Therefore, 
\begin{align*}
\p_t u^n +    \frac32  u^n  \p_x u^n  &=  n^{-s } \sin(n x-t) + \frac32 \left(  \frac23 n^{-1} + n^{-s  } \cos(n x-t)\right)(- n^{1 -s  } \sin(n x-t))  
\\
& = -
\frac32  n^{1 -2s } \cos(n x-t)\sin(n x-t).
\end{align*}
Using $\cos(x)\sin(x) = \frac12 \sin(2x) $ we obtain
\begin{align}
\| R \|  _{B^{\gamma} _{2,\infty}}  
& \lesssim
 n^{1 -2s  }   n^{\gamma }  +  n^{1 - s }   n^{\gamma -2  }  \lesssim n^{-\alpha},
\end{align}
where
\begin{equation}
\alpha = 
\begin{cases}
s+1-\g \ \ \ \ \ \ \text{if} \ \ \ s \geq 2 \\ 2s-1-\g \ \ \ \ \ \text{if} \ \ \ s<2.
\end{cases}•
\end{equation}•
\end{proof} From the estimate in inequality \eqref{estimate_id},  the solution  size estimate 
$
\| u_n (t)\| _{B^{s_2}_{2,\infty}}  \le 2\| u_n (0 )\| _{B^{s_2}_{2,\infty}} ,
$
 and repeated application of the algebra property, we have for any $s_2 > s$ 
\begin{align}
\| \mathcal{E} (t)\| _{B^{s_2 }_{2,\infty}}  \lesssim n^{ s_2 - s  } .
\end{align}
By applying the interpolation lemma for Besov spaces, we obtain 
\begin{align}\label{error_est_last}
\| \mathcal{E} (t)\| _{B^{s  }_{2,\infty}}  
&\le 
 \| \mathcal{E} (t)\| _{B^{\gamma  }_{2,\infty}} ^{\frac{s_2 - s}{s_2 - \gamma} } \| \mathcal{E} (t)\| _{B^{s_2  }_{2,\infty}}  ^{\frac{s - \gamma}{s_2 - \gamma} }  \lesssim
n^{-\beta},
\end{align} 
where 
\begin{equation}
\beta = 
\begin{cases}
\frac{s_2-s}{s_2-\g} \ \ \ \ \ \ \ \ \   \ \ \ \text{if} \ \ \ s \geq 2 \\ \frac{(s_2-s)(s-1)}{s_2-\g} \ \ \ \ \ \text{if} \ \ \ s<2.
\end{cases}•
\end{equation}•

\subsection{Nonuniform convergence}
We now show that $u_n$ and $v_n$ stay bounded away from each other for any $ t \neq 0$. 
\begin{align}
\| u_n - v_n \|  _{B^{s  }_{2,\infty}}  \ge \| u^n -v^n  \|  _{B^{s  }_{2,\infty}}  - \| u_n -u^n  \|  _{B^{s  }_{2,\infty}}   -\| v^n - v_n \|  _{B^{s  }_{2,\infty}} .
\end{align}
From \eqref{error_est_last} the second two terms tend towards zero as $n \rightarrow \infty$, therefore, 
\begin{align*}
\liminf_{n\rightarrow \infty} \| u_n - v_n \|  _{B^{s  }_{2,\infty}}  &\ge\liminf_{n\rightarrow \infty}  \| u^n -v^n  \|  _{B^{s  }_{2,\infty}}  
\\ & \ge
\liminf_{n\rightarrow \infty}  n^{-s  } \|   \cos(n x-t)- \cos(n x+t )  \|  _{B^{s  }_{2,\infty}}    
\\ & \approx 
\liminf_{n\rightarrow \infty}  n^{-s  }\|\sin(n  x)   \|  _{B^{s  }_{2,\infty}}    |\sin(t )|
 \approx 
 |\sin(t)  |  >0.
\end{align*}

%
%
\section{Analyticity of Solutions}

In this section, we search for real analytic solutions of the periodic Cauchy problem for the FW equation \eqref{fw}.  The classical Cauchy-Kowalevski theorem does not apply to equation \eqref{fw}.  However, a contraction argument on a scale of Banach spaces can be used for the nonlocal form of this equation \eqref{fwnl} to prove Theorem \ref{anal}.

\textbf{Remark 6.1:}  {\it The contraction argument used in the proof of Theorem \ref{anal} is on a decreasing scale of Banach spaces.  We give the statement of this abstract theorem here as stated in \cite{bao} and \cite{nish} for the convenience of the reader:

Given a Cauchy problem 
\begin{equation}
\label{ivp1}
\p_tu=F(t,u(t)), \ \ \ \ u(0)=0,
\end{equation}
and a decreasing scale of Banach spaces $\{X_s\}_{0<s<1}$ so that for any $s'<s$ we have $X_s \subset X_{s'}$ and $|||\cdot|||_{s'} \leq |||\cdot|||_{s}$, let $T$, $R$, and $C$ be positive numbers, suppose that $F$ satisfies the following conditions:
\begin{enumerate}
\item  If for $0<s'<s<1$ the function $t \mapsto u(t)$ is holomorphic in $|t|<T$ and continuous on $|t| \leq T$ with values in $X_s$ and $\sup_{|t| \leq T}|||u(t)|||_s<R$, then $t \mapsto F(t,u(t))$ is a holomorphic function on $|t|<T$ with values in $X_{s'}$.

\item  For any $0<s'<s \leq 0$ and any $u,v \in X_s$ with $|||u|||_s<R$, $|||v|||_s<R$,
\[
\sup_{|t| \leq T}|||F(t,u)-F(t,v)|||_{s'} \leq \frac{C}{s-s'}|||u-v|||_s.
\]

\item  There exists $M>0$ such that for any $0<s<1$,
\[
\sup_{|t| \leq T}|||F(t,0)|||_s \leq \frac{M}{1-s}.
\]
\end{enumerate}•
Then there exists a $T_0 \in (0,T)$ and a unique function $u(t)$, which for every $s \in (0,1)$ is holomorphic in $|t| < (1-s)T_0$ with values in $X_s$, and is a solution to the initial value problem \eqref{ivp1}.}

For $s>0$ let $E_s$ be defined as 
\[
E_s=\left\{u \in C^\infty(\ci) \ : \ |||u|||_s=\sup_{k>0}\frac{\|\p_x^ku\|_{B_{2,1}^{3/2}}s^k}{k!/(k+1)^2}<\infty \right\}.
\]
Let $X_s$ be the space $E_s$.  We use the decreasing scale of Banach spaces $X_s$ with $s>0$ for the contraction argument.  We take note that the above space is similar to the one introduced in \cite{hm4} where $B_{2,1}^{3/2}$ is replaced by $H^2$.  It is not difficult to check that $E_s$ equipped with the norm $|||\cdot|||_s$ is a Banach space and that for any $0<s'<s$, $E_s$ is continuously included in $E_{s'}$ with
\[
|||u|||_{s'} \leq |||u|||_s.
\]
We now include three properties of the spaces $E_s$ that we shall use in the proof of Theorem \ref{anal}.  Since we have the following from Besov spaces
\[
\|fg\|_{B_{2,1}^{3/2}} \leq C(\|f\|_{B_{2,1}^{3/2}}\|g\|_{B_{2,1}^{1/2}}+\|g\|_{B_{2,1}^{3/2}}\|f\|_{B_{2,1}^{1/2}})
\]
and
\[
\|\Lambda^\s\p_x^kf\|_{B_{2,1}^{3/2}} \leq C \|f\|_{B_{2,1}^{3/2+k+\s}}
\]
for $\s \in \rr$ and $k$ a non-negative integer, these three properties follow directly from analagous ones found in \cite{hm4} just by replacing $H^2$ with $B_{2,1}^{3/2}$.  We therefore omit the details for the sake of brevity and refer the reader to \cite{hm4}.

The first may be seen as a ring property for $E_s$ when $0<s<1$:  For any $u,v \in E_s$ there is a constant $c>0$ that is independent of $s$ such that
\[
|||uv|||_s \leq c|||u|||_s|||v|||_s.
\]
The second property is an estimate for the differential operator $\p_x$:  For $0<s'<s<1$ we have 
\[
|||u_x|||_{s'} \leq \frac{C}{s-s'}|||u|||_s
\]
for all $u \in E_s$.  The third property is a pair of estimates for the operator $\Lambda^{-2}=(1-\p_x^2)^{-1}$:  For $0<s<1$ we have
\[
|||\Lambda^{-2}u|||_{s'} \leq |||u|||_s \ \ \ \ \text{and} \ \ \ \ |||\Lambda^{-2}\p_xu|||_{s'} \leq |||u|||_s
\]
for all $u \in E_s$.  Now we proceed to the proof of Theorem \ref{anal}.

\begin{proof}
Note that it is sufficient to verify conditions (1) and (2) in the abstract Cauchy-Kowalevski theorem for
\[
F(u)=-\frac32uu_x+\Lambda^{-2}\p_xu
\]
since $F(u)$ does not depend on $t$ explicitly.  We find that for $0<s'<s<1$ we have
\[
|||F(u)|||_{s'} \leq \frac{C}{s-s'}|||u|||_s^2+|||u|||_s,
\]
hence condition (1) holds.  To verify the second condition we estimate that
\[
|||F(u)-F(v)|||_{s'} \leq C_R\left(|||u-v|||_s+\frac{1}{s-s'}|||u-v|||_s\right)
\]
for any $u,v \in E_s$.  Therefore, condition (2) is satisfied.
\end{proof}

%
%
%
%
\section{Blow-up Criterion}
Here we prove Theorem \ref{blowup}.
\begin{proof}
Applying $\D_q$ to \eqref{fwnl} yields
\begin{equation}
\left(\p_t+\frac32u\p_x\right)\D_qu=\left[\frac32u,\D_q\right]\p_xu+\D_q\Lambda^{-1}\p_xu.
\end{equation}•
From (2.54) on p. 112 in \cite{bcd}, we have that
\begin{equation}
\label{est1}
\left\|2^{sq}\left\|\left[\frac32u,\D_q\right]\p_xu\right\|_{L^2}\right\|_{l^r} \leq C\|u_x\|_{L^\infty}\|u\|_{B_{2,r}^s}.
\end{equation}•
We also have that
\begin{equation}
\label{est2}
\|\Lambda^{-1}\p_xu\|_{B_{2,r}^s} \leq \|u\|_{B_{2,r}^s}.
\end{equation}•
Going along the lines of the proof of Proposition A.1 in \cite{danchin} and by using \eqref{est1} and \eqref{est2} we obtain
\begin{equation}
\label{est3}
\|u\|_{B_{2,r}^s} \leq \|u_0\|_{B_{2,r}^s}+C\int_0^t[1+\|u_x\|_{L^\infty}]\|u\|_{B_{2,r}^s}d\tau
\end{equation}•
Applying Gronwall's inequality to \eqref{est3} yields
\begin{equation}
\label{est4}
\|u(t)\|_{B_{2,r}^s} \leq \|u_0\|_{B_{2,r}^s}\text{exp}\left[C\int_0^t[1+\|u_x\|_{L^\infty}]d\tau\right].
\end{equation}•
Assume that $T^*$ is the maximal time of existence of the solution to the Cauchy problem to \eqref{fwnl}.  If $T^*< \infty$ we claim that
\begin{equation}
\label{est5}
\int_0^{T^*}\|u_x\|_{L^\infty}d\tau = +\infty.
\end{equation}•
We prove the claim \eqref{est5} by contradiction.  If \eqref{est5} is invalid, then from \eqref{est4} we know that $\|u(T^*)\|_{B_{2,r}^s}<\infty$ which contradicts with the fact that $T^*$ is the maximal time of existence to the Cauchy problem to \eqref{fwnl}.  We have concluded the proof of Theorem \ref{blowup}.
\end{proof}

\end{document}